\documentclass{article}

\usepackage{arxiv}

\usepackage[utf8]{inputenc} 
\usepackage[T1]{fontenc}    
\usepackage{hyperref}       
\usepackage{url}            
\usepackage{booktabs}       
\usepackage{amsfonts}       
\usepackage{nicefrac}       
\usepackage{microtype}      
\usepackage{lipsum}

\usepackage{tikz}

\usetikzlibrary{matrix,arrows,shapes.geometric}

\usepackage[all,2cell,cmtip]{xy}

\usepackage{mathtools}

\usepackage{amssymb}

\usepackage{enumerate}

\usepackage[all]{xy}
\xyoption{2cell}
\xyoption{curve}
\UseTwocells
\input{diagxy}
\CompileMatrices      
\usepackage{tikz}
\usepackage{pdfpages}

\usepackage{mathrsfs}

\usepackage{tikz}

\usepackage{stmaryrd}

\usepackage{amsthm}

\newtheorem{lema}{Lemma}

\newtheorem{dfn}{Definition}

\newtheorem{thm}{Theorem}

\newtheorem{pps}{Proposition}

\newtheorem{cor}{Corollary}

\newtheorem{rem}{Remark}

\title{Quillen Model Structures-Based Notions of Locality of Logics over Finite Models}

\author{
  Hendrick Maia
  \\
  \texttt{hendrickmaia@gmail.com} \\
}

\begin{document}
\maketitle

\begin{abstract}

Locality is a property of logics, based on Hanf’s and Gaifman’s theorems, and that was shown to be very useful in the context of finite model theory. In this paper I present a homotopic variation for locality, namely a Quillen model category-based framework for locality under $k$-logical equivalence, for every primitive-positive sentence of quantifier-rank $k$. 

\end{abstract}


\section{Introduction}

Locality is a property of logics, whose origins lie in the works of Hanf [14] and Gaifman [12], having their utility in the context of finite model theory. Such a property is quite useful in proofs of inexpressibility, but it is also useful in establishing normal forms for logical formulas. 

There are generally two forms of locality: (i') if two structures $\mathfrak{A}$ and $\mathfrak{B}$ realize the same multiset of types of neighborhoods of radius $d$, then they agree on a given sentence $\Phi$. Here $d$ depends only on $\Phi$; (ii')  if the $d$-neighborhoods of two tuples $\vec{a}_1$ and $\vec{a}_2$ in a structure $\mathfrak{A}$ are isomorphic, then $\mathfrak{A} \models \Phi(\vec{a}_1) \Leftrightarrow \Phi(\vec{a}_2)$. Again, $d$ depends on $\Phi$, and not on $\mathfrak{A}$. Form (i') originated from Hanf's works [14]. Form (ii') came from Gaifman's theorem [12]. Before proceeding, I will establish some notation. 

\noindent \textbf{Notations:} All structures here are finite, whose vocabularies are finite sequences of relation symbols $\sigma = \langle R_1,...,R_l \rangle$. A $\sigma$-structure $\mathfrak{A}$ consists of a finite universe $A$ and an interpretation of each $p_i$-ary relation symbol $R_i$ in $\sigma$ as $R^{\mathfrak{A}}_i \subseteq A^{p_i}$. 

Given two structures $\mathfrak{A}$ and $\mathfrak{B}$ of a relational vocabulary $\sigma$, a homomorphism between them is a mapping $h : \mathfrak{A} \rightarrow \mathfrak{B}$ such that for each constant symbol $c$ in $\sigma$, we have $h(c^{\mathfrak{A}}) = c^{\mathfrak{B}}$, and for each $k$-ary relation symbol $R$ and a tuple $(a_1,...,a_k) \in R^{\mathfrak{A}}$, the tuple $(h(a_1),...,h(a_k))$ is in $R^{\mathfrak{B}}$. A bijective homomorphism $h$ whose inverse is also a homomorphism is called an \emph{isomorphism}.
If there is an isomorphism between two structures $\mathfrak{A}$ and $\mathfrak{B}$, we say that they are isomorphic, and we write $\mathfrak{A} \cong \mathfrak{B}$.

$\mathbf{STRUCT}[\sigma]$ denotes the category of $\sigma$-structures. I shall use the notation $\sigma_n$ for $\sigma$ expanded with $n$ constant symbols.  

The quantifier-rank of a formula $\Phi$ is the maximal nesting depth of quantifiers in $\Phi$. 

Given a structure $\mathfrak{A}$, its \emph{Gaifman graph} $\mathcal{G}(\mathfrak{A})$ is defined as $\langle A,E \rangle$ where $(a,b)$ is in $E$ if, and only if there is a tuple $\vec{c} \in R^{\mathfrak{A}}_i$ for some $i$ such that both $a$ and $b$ are in $\vec{c}$. The distance $d(a,b)$ is defined as the length of the shortest path from $a$ to $b$ in $\mathcal{G}(\mathfrak{A})$; we assume $d(a,a) = 0$. If $\vec{a} = (a_1,...,a_n)$, then $d(\vec{a},b) = \mathrm{min}_id(a_i,b)$. Given $\vec{a}$ over $A$, its \emph{$r$-ball} $B_r^{\mathfrak{A}}(\vec{a})$ is $\{b \in A \mid d(\vec{a},b) \leq r\}$. If $|\vec{a}| = n$, its \emph{$r$-neighborhood} $N_r^{\mathfrak{A}}(\vec{a})$ is defined as a $\sigma_n$-structure
$$\langle B_r^{\mathfrak{A}}(\vec{a}), R_1^{\mathfrak{A}} \cap B_r^{\mathfrak{A}}(\vec{a})^{p_1},..., R_l^{\mathfrak{A}} \cap B_r^{\mathfrak{A}}(\vec{a})^{p_l}, a_1,...,a_n \rangle.$$
Note that for any isomorphism $h:N_r^{\mathfrak{A}}(\vec{a}) \rightarrow N_r^{\mathfrak{B}}(\vec{b})$ it must be the case that $h(\vec{a}) = \vec{b}$.

Given a tuple  $\vec{a} = (a_1,...,a_n)$ and an element $c$, we write  $\vec{a}c$ for the tuple $(a_1,...,a_n,c)$. 

An $m$-ary query, $m \geq 0$, on $\sigma$-structures, is a mapping $Q$ that associates with each structure $\mathfrak{A}$ a subset of $A^m$, such that $Q$ is closed under isomorphism: if $\mathfrak{A} \cong \mathfrak{B}$ via isomorphism $h : A \rightarrow B$, then $Q(\mathfrak{B}) = h(Q(\mathfrak{A}))$.

We write $\mathfrak{A} \equiv_k \mathfrak{B}$ if $\mathfrak{A}$ and $\mathfrak{B}$ agree on all FO sentences of quantifier-rank up to $k$, and $(\mathfrak{A},\vec{a}) \equiv_k (\mathfrak{B},\vec{b})$ if $\mathfrak{A} \models \Phi(\vec{a}) \Leftrightarrow \mathfrak{B} \models \Phi(\vec{b})$ for every FO formula $\Phi(\vec{x})$ of quantifier rank up to $k$ ($k$-logical equivalence). It is well known that $\mathfrak{A} \equiv_k \mathfrak{B}$ if, and only if, the duplicator has a winning strategy in the $k$-round Ehrenfeucht-Fra\"iss\'e game on $\mathfrak{A}$ and $\mathfrak{B}$, and $(\mathfrak{A},\vec{a}) \equiv_k (\mathfrak{B},\vec{b})$ if, and only if, the duplicator has a winning strategy in the $k$-round Ehrenfeucht-Fra\"iss\'e game on $\mathfrak{A}$ and $\mathfrak{B}$ starting in position $(\vec{a},\vec{b})$.

\qed

There is no doubt about the usefulness of the notion of locality, which as seen applies to a huge number of situations. However, there is a deficiency in such a notion: all versions of the notion of locality refer to isomorphism of neighborhoods, which is a fairly strong property. For example, where structures simply do not have sufficient isomorphic neighborhoods, versions of the notion of locality obviously cannot be applied. So the question that immediately arises is: would it be possible to weaken such a condition, and maintain Hanf / Gaifman-localities?

Arenas, Barceló and Libkin [1] establish a new condition for the notions of locality, weakening the requirement that neighborhoods should be isomorphic, establishing only the condition that they must be indistinguishable in a given logic. That is, instead of requiring $N_d(\vec{a}) \cong N_d(\vec{b})$, you should only require $N_d(\vec{a}) \equiv_k N_d(\vec{b})$, for some $k \geq 0$. Using the fact that logical equivalence is often captured by Ehrenfeucht – Fra{\"i}ssé games, the authors formulate a game-based framework in which logical equivalence-based locality can be defined. Thus, the notion defined by the authors is that of \emph{game-based locality}.

Note that the intuitive point from which the authors start is the idea of \emph{neighborhood indistinguishability}. Thus, the intuition behind the notion of game-based locality is to describe the indistinguishability of neighborhoods in terms of \emph{winning game strategies}. To achieve the necessary generalization, Arenas, Barceló and Libkin define an abstract view of the games that characterize the expressiveness of logics that are local under isomorphism. The basic idea is as follows: in each round the duplicator has a set of functions (tactics) that will determine his responses to possible moves by the spoiler. In order to capture this idea, the authors define the abstract notion of \emph{agreement} ([1] p.5).  

\begin{dfn}
	An \emph{agreement} $\mathfrak{F}$ assigns to each pair $A,B$ of finite subsets of $U$ a collection 
	
	$$\mathfrak{F}(A,B) = \{\mathcal{F}_1(A,B),...,\mathcal{F}_m(A,B)\},$$ 
	
	\noindent where each $\mathcal{F}_i(A,B)$ is a nonempty collection of partial functions $f:A \rightarrow B$. We call the sets $\mathcal{F}_i(A,B)$ \emph{tactics}. 
	
	\noindent The $\mathfrak{F}$-game on $(\mathfrak{A},\vec{a}_0)$ and $(\mathfrak{B},\vec{b}_0)$ is played as follows. Suppose after $i$ rounds the position is $(\vec{a}_0\vec{a},\vec{b}_0\vec{b})$ (before the game starts, the tuples $\vec{a},\vec{b}$ are empty). Then, in round $i+1$:
	
	\begin{enumerate}
		\item The spoiler chooses a structure, $\mathfrak{A}$ or $\mathfrak{B}$. Below we present the moves assuming he chose $\mathfrak{A}$, the case of $\mathfrak{B}$ is symmetric.
		\item The duplicator chooses a tactic $\mathcal{F}(A,B) \in \mathfrak{F}(A,B)$.
		\item The spoiler chooses a partial function $f \in \mathcal{F}(A,B)$ and an element $a \in \mathrm{dom}(f)$; the game continues from the position $(\vec{a}_0\vec{a}a,\vec{b}_0\vec{b}f(a))$.
	\end{enumerate}

	The duplicator wins after $k$-rounds if both $\mathfrak{F}(A,B)$ and $\mathfrak{F}(B,A)$ are non-empty, and the final position defines a partial isomorphism between $(\mathfrak{A},\vec{a}_0)$ and $(\mathfrak{B},\vec{b}_0)$. If the duplicator has a winning strategy for the $k$-round game, we write $(\mathfrak{A},\vec{a}_0) \equiv_k^{\mathfrak{F}} (\mathfrak{B},\vec{b}_0)$.
\end{dfn}

The notions of \emph{game for a logic} and \emph{capture} are also defined:

\begin{dfn}
	
	Given an agreement $\mathfrak{F}$, we say that the $\mathfrak{F}$-game is a \emph{game for a logic $\mathcal{L}$} if there exists a partition $\{\mathcal{L}_0, \mathcal{L}_1,...\}$ of the formulae in $\mathcal{L}$ such that for every $k \geq 0$, there exists $k' \geq 0$ with the property that 
	
	$$(\mathfrak{A},\vec{a}_0) \equiv_{k'}^{\mathfrak{F}} (\mathfrak{B},\vec{b}_0) \textrm{ implies } (\mathfrak{A} \models \varphi(\vec{a}) \Leftrightarrow \mathfrak{B} \models \varphi(\vec{b})), \textrm{ } \textrm{ for all } \varphi \in \mathcal{L}_k.$$
	\noindent If the converse holds as well, that is, for every $k' \geq 0$ there exists $k \geq 0$ such that, $(\mathfrak{A},\vec{a}_0) \equiv_{k'}^{\mathfrak{F}} (\mathfrak{B},\vec{b}_0)$, whenever $\mathfrak{A} \models \varphi(\vec{a}) \Leftrightarrow \mathfrak{B} \models \varphi(\vec{b})$ for every $\varphi \in \mathcal{L}_k$, then we say that the $\mathfrak{F}$-game \emph{captures} $\mathcal{L}$.
	
\end{dfn}

In the following, keep in mind that $\mathcal{L}_k$ will always be associated with the set of $\mathcal{L}$-formulae of quantifier rank $k$. Furthermore, if $\mathfrak{F}$ is a game for a logic $\mathcal{L}$, and $\mathfrak{F}'$-games capture $\mathcal{L}$, then for every $k \geq 0$ there exists $k' \geq 0$ such that

$$(\mathfrak{A},\bar{a}) \equiv_{k'}^{\mathfrak{F}} (\mathfrak{B},\bar{b}) \Leftrightarrow (\mathfrak{A},\bar{a}) \equiv_k^{\mathfrak{F}'} (\mathfrak{B},\bar{b}).$$

We will now see how Arenas, Barceló and Libkin weaken in [1] the requirement that neighborhoods should be isomorphic. For $d, \ell \geq 0$, Arenas, Barceló and Libkin use the notation $(\mathfrak{A}, \bar{a}) \rightleftarrows_{d,\ell}^{\mathfrak{F}} (\mathfrak{B},\bar{b})$ if there exists a bijection $f:A \rightarrow B$ such that $N_d^{\mathfrak{A}}(\bar{a},c) \equiv_\ell^{\mathfrak{F}} N_d^{\mathfrak{B}}(\bar{b},f(c)), \textrm{ } \textrm{ for every } c \in A.$

\begin{dfn}
	
	An agreement $\mathfrak{F}$ is: 
	\begin{itemize}
		\item \emph{Hanf-local} if for every $k,m \in \mathbb{N}$, there exists $d,\ell \in \mathbb{N}$ such that for every two structures $\mathfrak{A}, \mathfrak{B}$, $\vec{a} \in A^m$ and $\vec{b} \in B^m$,
		
		$$(\mathfrak{A},\vec{a}) \rightleftarrows_{d,\ell}^{\mathfrak{F}} (\mathfrak{B},\vec{b}) \Rightarrow (\mathfrak{A},\vec{a}) \equiv_{k}^{\mathfrak{F}} (\mathfrak{B},\vec{b}).$$
		
		\item \emph{Gaifman-local} if for every $k,m \in \mathbb{N}$, there exists $d,\ell \in \mathbb{N}$ such that
		for every two structures $\mathfrak{A}, \mathfrak{B}$, $\vec{a} \in A^m$ and $\vec{b} \in B^m$,
		
		$$\mathfrak{A} \equiv_{\ell}^{\mathfrak{F}} \mathfrak{B} \textrm{ and } N_d^{\mathfrak{A}}(\vec{a}) \equiv_\ell^{\mathfrak{F}} N_d^{\mathfrak{B}}(\vec{b}) \Rightarrow (\mathfrak{A},\vec{a}) \equiv_{k}^{\mathfrak{F}} (\mathfrak{B},\vec{b}).$$
		
		\item \emph{weakly-local} if for every $k,m \in \mathbb{N}$, there exist $d,\ell \in \mathbb{N}$ such
		that for every structure $\mathfrak{A}$, $\vec{a}$ and $\vec{b} \in A^m$, 
		
		$$N_d^{\mathfrak{A}}(\vec{a}) \equiv_\ell^{\mathfrak{F}} N_d^{\mathfrak{B}}(\vec{b}) \textrm{ and } B_d^{\mathfrak{A}}(\vec{a}) \cap B_d^{\mathfrak{B}}(\bar{b}) = \emptyset \Rightarrow (\mathfrak{A},\vec{a}) \equiv_{k}^{\mathfrak{F}} (\mathfrak{B},\vec{b}).$$
	\end{itemize}
		
\end{dfn}

Although quite promising as well as easy to apply, the game-based framework (used to define locality under logical equivalence) has the following problem: if a logic $\mathcal{L}$ is local (Hanf-, or Gaifman-, or weakly)
under isomorphisms, and $\mathcal{L}'$ is a sub-logic of $\mathcal{L}$, then $\mathcal{L}'$ is local as well. The same, however, is not true for game-based locality: properties of games guaranteeing locality need not be preserved if one passes to weaker games [1].

The question that immediately arises is: is it possible to define the notion of locality under logical equivalence without resorting to game-based frameworks? The purpose of this paper is to provide a partial answer to this question. As we will see, at least for positive primitive-sentences, the answer is yes. 


{\bf Overview of the paper:} In {\bf Section 2} we present the results of model categories that will be used.  {\bf Section 3} we present the results about cores that will be needed. {\bf Section 4} we define the category $\mathbf{STRUCT}[\sigma_n]_{(d,0)}^{\widetilde{\mathcal{T}(\sigma_n)}}$ and the apparatus necessary to present the main result. {\bf Section 5} is devoted to the main contributions of the present work (derivated from [18]): In stating my main result -- Theorem 3 below --. I switch from talking about formulas to sentences (i.e., formulas without free variables)\footnote{This is merely a matter of convenience; Theorem 3 remains valid when stated more generally for formulas instead of sentences.}.

\section{Quillen Model Categories}

In this section we introduce the concept of a Quillen model category. 

\begin{dfn}

Given a commutative square diagram of the following form

\begin{equation}
\xymatrix{& A \ar[d]_i \ar[r]^{f} & X \ar[d]^p \\ & B  \ar[r]_{g} & Y,}
\end{equation}

\end{dfn}

\noindent a \emph{lift} or \emph{lifting} in the diagram is a map $h:B \rightarrow X$ such that the resulting diagram with  five arrows commutes, i.e., such that $h \circ i=f$ and $p \circ h = g$.

\begin{dfn}
A model category is a category $\mathcal{C}$ with three distinguished classes of maps:

\begin{enumerate}

\item weak equivalences ($\xrightarrow{\sim}$);
\item fibrations ($\twoheadrightarrow$); and
\item cofibrations ($\hookrightarrow$).

\end{enumerate}

\noindent each of which is closed under composition and contains all identity maps. A map which  is  both  a fibration (resp. cofibration) and  a  weak equivalence is  called an \emph{acyclic fibration} (resp. \emph{acyclic cofibration}). We require the following axioms:

\noindent \emph{\textbf{MC1}} Finite limits and colimits exist in $\mathcal{C}$;

\noindent \emph{\textbf{MC2}} If $f$ and $g$ are maps in $\mathcal{C}$ such that $g \circ f$ is defined and if two of the three maps $f,g,g \circ f$ are weak equivalences, then so is the third.

\noindent \emph{\textbf{MC3}} If $f$ is a retract of $g$ \emph{(see [3] 2.6)} and $g$ is a  fibration, cofibration, or a weak equivalence, then so is $f$.

\noindent \emph{\textbf{MC4}} Given a commutative diagram of the form (1), a lift exists in the diagram in either of the following two situations: (i) $i$ is a cofibration and $p$ is an acyclic fibration, or (ii) $i$ is an acyclic cofibration and $p$ is a fibration.

\noindent \emph{\textbf{MC5}} Any map $f$ can be factored in two ways: (i) $f = p \circ i$, where $i$ is a cofibration and $p$ is an acyclic fibration, and (ii) $f = p \circ i$, where $i$ is an acyclic cofibration and $p$ is a fibration.

\end{dfn}

\noindent By \textbf{MC1} and ([3], 2.25), a model category $\mathcal{C}$ has both an initial object $\emptyset$ and a terminal  object $*$. An object $A \in \mathcal{C}$ is said to be \emph{cofibrant} if $\emptyset \rightarrow A$ is a cofibration and \emph{fibrant} if $A \rightarrow *$ is a fibration. 

\begin{dfn}[Lifting Properties]
A map $i:A \rightarrow B$ is said to have the left lifting property \emph{(LLP)} with respect to another map $p:X \rightarrow Y$ and $p$ is said to have the right lifting property \emph{(RLP)} with respect to $i$ if a lift exists in any diagram of the form \emph{(1)}.
\end{dfn}

\begin{pps}
Let $\mathcal{C}$ be a model category.

\begin{enumerate}
\item The  cofibrations in $\mathcal{C}$ are  the  maps  which  have  the  LLP  with  respect to acyclic  fibrations.
\item The acyclic cofibrations in $\mathcal{C}$ are the maps which have the LLP with respect to cofibrations.
\item The fibrations in $\mathcal{C}$ are the maps which have the RLP with respect to acyclic cofibrations.
\item The  acyclic fibrations  in $\mathcal{C}$ are  the  maps  which  have  the  RLP with respect to cofibrations.
\end{enumerate}
\end{pps}

\begin{proof}
([3], p.87).
\end{proof}

\begin{pps}
Let $\mathcal{C}$ be a model category.
\begin{enumerate}
\item The class of cofibrations in $\mathcal{C}$ is stable under cobase change \emph{(see [3], 2.16)}. 
\item The class of acyclic cofibrations in $\mathcal{C}$ is stable under cobase change.
\item The class of  fibrations in $\mathcal{C}$ is stable under base change \emph{(see [3], 2.23)}.
\item The class of acyclic fibrations in $\mathcal{C}$ is stable under base change.
\end{enumerate}
\end{pps}

\begin{proof}
([3], p. 88).
\end{proof}

\subsection{Homotopy Relations on Maps}

\subsubsection{Cylinder Objects and Left Homotopy}

In this subsubsection $\mathcal{C}$ is some fixed model category, and $A$ and $X$ are objects of $\mathcal{C}$.

\begin{dfn}[Cylinder objects]
A cylinder object for $A$ is an object $A \wedge I$ of $\mathcal{C}$ together with a diagram \emph{(\textbf{MC1}, [3], 2.15)}:

$$A \coprod A \xrightarrow{i} A \wedge I \xrightarrow{\sim} A$$

\noindent which factors the folding  map $\mathrm{id}_A + \mathrm{id}_A: A \coprod A \rightarrow A$ \emph{(see [3], 2.15)}. A cylinder object $A \wedge I$ is called

\begin{enumerate}
\item a \emph{good cylinder object}, if $A \coprod A \rightarrow A \wedge I$ is a cofibration; and
\item a \emph{very good cylinder object}, if in addition the map $A \wedge I \rightarrow A$ is a (necessarily acyclic) fibration.
\end{enumerate}

\end{dfn}

\noindent If $A \wedge I$ is a cylinder object for $A$, we will denote the two structure maps $A \rightarrow A \wedge I$ by $i_0 = i \cdot \mathrm{in}_0$ and $i_1 = i \cdot \mathrm{in}_1$ (cf. [3], 2.15).

\begin{lema}
If $A$ is cofibrant and $A \wedge I$ is a good cylinder object for $A$, then  the maps $i_0,i_1:A \rightarrow A \wedge I$ are acyclic cofibrations.
\end{lema}

\begin{proof}
 ([3], pp. 89-90).
\end{proof}

\begin{dfn}
Two maps $f,g:A \rightarrow X$ in $\mathcal{C}$ are said to be left  homotopic (written $f \sim_l g$) if there exists a cylinder object $A \wedge I$ for $A$ such that the sum map $f+g:A \coprod A \rightarrow X$ \emph{(see [3], 2.15)} extends to a map $H:A \wedge I \rightarrow X$, i.e. such that there exists a map $H:A \wedge I \rightarrow X$ with $H(i_0+i_1)=f+g$. Such a map $H$ is said to be a \emph{left homotopy} from $f$ to $g$ (via the cylinder object $A \wedge I$). The left homotopy is said to be \emph{good} (resp. \emph{very good}) if $A \wedge I$ is a good (resp. very good) cylinder object for $A$.
\end{dfn}

\begin{lema}
If $f \sim_l g:A \rightarrow X$, then there  exists a  good left  homotopy  from $f$ to $g$. If in addition $X$ is fibrant, then there exists a very good left homotopy from $f$ to $g$.
\end{lema}

\begin{proof}
 ([3], p. 90).
\end{proof}

\begin{lema}
If $A$ is cofibrant, then $\sim_l$ is an equivalence relation on $\mathrm{Hom}_{\mathcal{C}}(A,X)$.
\end{lema}

\begin{proof}
 ([3], p. 91).
\end{proof}

Let $\pi^l(A,X)$ denote the set of equivalence classes of $\mathrm{Hom}_{\mathcal{C}}(A,X)$ under the equivalence relation generated by left homotopy.

\begin{lema}
If $A$ is cofibrant and $p:Y \rightarrow X$ is an acyclic fibration, then composition with $p$ induces a bijection:
$$p_*:\pi^l(A,Y) \rightarrow \pi^l(A,X), \textrm{ } [f] \mapsto [p \circ f].$$
\end{lema}

\begin{proof}
 ([3], pp. 91-92).
\end{proof}

\begin{lema}
Suppose that $X$ is fibrant,  that $f$ and $g$ are left homotopic maps $A \rightarrow X$, and that $h:A' \rightarrow A$ is a map. Then $f \circ h \sim_l g \circ h$.
\end{lema}

\begin{proof}
 ([3], p. 92).
\end{proof}

\begin{lema}
If $X$ is fibrant, then the composition in $\mathcal{C}$ induces a map:
$$\pi^l(A',A) \times \pi^l(A,X) \rightarrow \pi^l(A',X), \textrm{ } ([h],[f]) \mapsto [f \circ h].$$
\end{lema}

\begin{proof}
 ([3], p. 92).
\end{proof}

\subsubsection{Path Objects and Right Homotopies}


\begin{dfn}[Path objects]
A path object for $X$ is an object $X^I$ of $\mathcal{C}$ together with a diagram:

$$X \xrightarrow{\sim} X^I \xrightarrow{p} X \times X$$

\noindent which factors the diagonal map $(\mathrm{id}_X,\mathrm{id}_X): X \rightarrow X \times X$. A path object $X^I$ is called

\begin{enumerate}
\item a \emph{good path object}, if $X^I \rightarrow X \times X$ is a fibration; and
\item a \emph{very good path object}, if in addition the map $X \rightarrow X^I$ is a (necessarily acyclic) cofibration.
\end{enumerate}

\end{dfn}

\noindent By \textbf{MC5}, at least one  very good path object exists  for $X$. An object $X$ of $\mathcal{C}$ might have many path objects associated to it, denoted $X^I,X^{I'},...,$ etc. We denote the two maps $X^I \rightarrow X$ by $p_0 = pr_0 \cdot p$ and $p_1 = pr_1 \cdot p$ ([3], cf. 2.22).

\begin{lema}
If $X$ is fibrant and $X^I$ is a good  path  object for $X$, then  the maps $p_0,p_1:X^I \rightarrow X$ are acyclic fibrations.
\end{lema}

\begin{proof}
Dual
\end{proof}

\begin{dfn}
Two maps $f,g:A \rightarrow X$ are said to be right homotopic (written $f \sim_r g$) if there exists a path object $X^I$ for $X$ such that the product map $(f,g):A \rightarrow X \times X$ lifts to a map $H:A \rightarrow X^I$. Such a map $H$ is said to be a right homotopy from $f$ to $g$ (via the path object $X^I$). The right homotopy is said to be good (resp.very good) if $X^I$ is a good (resp. very good) path object for $X$.
\end{dfn}

\begin{lema}
If $f \sim_r g:A \rightarrow X$, then there  exists a  good right homotopy from $f$ to $g$. If in addition $A$ is cofibrant, then there exists a very good right homotopy from $f$ to $g$.
\end{lema}

\begin{proof}
Dual
\end{proof}

\begin{lema}
If $X$ is fibrant, then $\sim_r$ is an equivalence relation on $\mathrm{Hom}_{\mathcal{C}}(A,X)$.
\end{lema}

\begin{proof}
Dual
\end{proof}

Let $\pi^r(A,X)$ denote the set of equivalence classes of $\mathrm{Hom}_{\mathcal{C}}(A,X)$ under the equivalence relation generated by right homotopy.

\begin{lema}
If $X$ is fibrant and $i:A \rightarrow B$ is an acyclic cofibration, then composition with $i$ induces a bijection:
$$i^*:\pi^r(B,X) \rightarrow \pi^r(A,X).$$
\end{lema}

\begin{proof}
Dual
\end{proof}

\begin{lema}
Suppose that $A$ is cofibrant,  that $f$ and $g$ are right homotopic maps from $A$ to $X$, and that $h:X \rightarrow Y$ is a map. Then $h \circ f \sim_r h \circ g$.
\end{lema}

\begin{proof}
Dual
\end{proof}

\begin{lema}
If $A$ is cofibrant, then the composition in $\mathcal{C}$ induces a map:
$$\pi^r(A,X) \times \pi^r(X,Y) \rightarrow \pi^r(A,Y).$$
\end{lema}

\begin{proof}
Dual
\end{proof}

\subsubsection{Relationship between Left and Right Homotopy}

\begin{lema}
Let $f,g:A \rightarrow X$ be maps.
\begin{enumerate}
\item If $A$ is cofibrant and $f \sim_l g$, then $f \sim_r g$.
\item If $X$ is fibrant and $f \sim_r g$, then $f \sim_l g$.
\end{enumerate}
\end{lema}

\begin{proof}
 ([3], p. 94).
\end{proof}

\noindent If $A$ is cofibrant  and $X$ is fibrant,  we  will  denote  the identical right homotopy and left homotopy equivalence relations on $\mathrm{Hom}_\mathcal{C}(A,X)$ by the symbol ''$\sim$'' and say that two maps related by this relation are homotopic. The set of equivalence classes with respect to this relation is denoted $\pi(A,X)$. 

\begin{lema}
Suppose  that $f:A \rightarrow X$ is  a  map  in $\mathcal{C}$ between  objects $A$ and $X$ which are both  fibrant and cofibrant. Then $f$ is a weak equivalence if and only if $f$ has a homotopy inverse, i.e., if and only if there exists a map $g:X \rightarrow A$ such that the composites $g \circ f$ and $f \circ g$ are homotopic to the respective identity maps.
\end{lema}

\begin{proof}
 ([3], pp. 94-95).
\end{proof}

\subsection{The Homotopy Category of a Model Category}

We begin by looking at the following six categories associated to $\mathcal{C}$.

\begin{itemize}
\item $\mathcal{C}_c$ - the full subcategory of $\mathcal{C}$ generated by the cofibrant objects in $\mathcal{C}$.
\item $\mathcal{C}_f$ - the full subcategory of $\mathcal{C}$ generated by the  fibrant objects in $\mathcal{C}$.
\item $\mathcal{C}_{cf}$ - the full subcategory of $\mathcal{C}$ generated by the objects of $\mathcal{C}$ which are both  fibrant and cofibrant.
\item $\pi \mathcal{C}_c$ - the category consisting of the cofibrant objects in $\mathcal{C}$ and whose morphisms are right homotopy classes of maps.
\item $\pi \mathcal{C}_f$ - the category consisting of  fibrant objects in $\mathcal{C}$ and whose morphisms are left homotopy classes of maps.
\item $\pi \mathcal{C}_{cf}$ - the category consisting of objects in $\mathcal{C}$ which are both  fibrant and cofibrant, and whose morphisms are homotopy classes of maps.
\end{itemize}

As pointed out in ([3], p. 96), these  categories  will  be  used  as  tools  in  defining  $\mathrm{Ho}(\mathcal{C})$  and  constructing  a canonical functor $\mathcal{C} \rightarrow \mathrm{Ho}(\mathcal{C})$. For each object $X$ in $\mathcal{C}$ we can apply \textbf{MC5} (i) to the map $\emptyset \rightarrow X$ and obtain an acyclic  fibration $p_X:QX \tilde{\twoheadrightarrow} X$ with $QX$ cofibrant. We can also apply \textbf{MC5} (ii) to the map $X \rightarrow *$ and obtain an acyclic cofibration $i_X:X \tilde{\hookrightarrow} RX$ with $RX$ fibrant. If $X$ is itself cofibrant, let $QX = X$; if $X$ is fibrant,let $RX = X$ (see [3], p. 96).

\begin{lema}
Given  a  map $f:X \rightarrow Y$ in $\mathcal{C}$ there  exists  a  map $\tilde{f}:QX \rightarrow QY$ such that the following  diagram commutes:

\begin{equation}
\xymatrix{& QX \ar[d]_{p_X}^\sim \ar[r]^{\tilde{f}} & QY \ar[d]_{p_Y}^\sim \\ & X  \ar[r]_{f} & Y.}
\end{equation}
\noindent The  map $\tilde{f}$ depends  up  to  left  homotopy  or  up  to  right  homotopy  only  on $f$, and is  a weak  equivalence  if  and only  if $f$ is.  If $Y$ is  fibrant, then $\tilde{f}$ depends  up to left homotopy  or up to right homotopy only on the left homotopy  class of $f$.
\end{lema}

\begin{proof}
 ([3], p. 96).
\end{proof}

As pointed out in ([3], 5.2. Remark), the uniqueness statements in Lemma 15 imply that if $f = \mathrm{id}_X$ then $\tilde{f}$ is right homotopic to $\mathrm{id}_{QX}$. Similarly, if $f:X \rightarrow Y$ and $g:Y \rightarrow Z$ and $h = g \circ f$, then $\tilde{h}$ is right homotopic to $\tilde{g} \circ \tilde{f}$. Hence we can define a functor $Q:\mathcal{C} \rightarrow \pi \mathcal{C}_c$ sending $X \rightarrow QX$ and $f:X \rightarrow Y$ to the right homotopy class $[\tilde{f}] \in \pi^r(QX,QY)$.

\begin{lema}
Given a map $f:X \rightarrow Y$ in $\mathcal{C}$ there exists a map $\bar{f}:RX \rightarrow RY$ such that the following diagram commutes:

\begin{equation}
\xymatrix{& X \ar[d]_{i_X}^\sim \ar[r]^{f} & Y \ar[d]_{i_Y}^\sim \\ & RX  \ar[r]_{\bar{f}} & RY.}
\end{equation}
\noindent The map $\bar{f}$ depends up to right homotopy or up to left homotopy only on $f$,and is a weak equivalence if and only if $f$ is. If $X$ is cofibrant, then $\bar{f}$ depends up to right homotopy or up to left homotopy  only on the right homotopy  class of $f$.

\end{lema}

\begin{proof}
Dual
\end{proof}

As pointed out in ([3], 5.4. Remark), the uniqueness statements in Lemma 16 imply that if $f=id_X$ then $\bar{f}$ is left homotopic to $\mathrm{id}_{RX}$. Moreover, if $f:X \rightarrow Y$ and $g:Y \rightarrow Z$ and $h=g \circ f$, then $\bar{h}$ is left  homotopic to  $\bar{g} \circ \bar{f}$, Hence we can define a functor $R:\mathcal{C} \rightarrow \pi \mathcal{C}_f$ sending $X \rightarrow RX$ and $f:X \rightarrow Y$ to the left homotopy class $[\bar{f}] \in \pi^l(RX,RY)$.

\begin{lema}
The restriction of the functor $Q:C \rightarrow \pi \mathcal{C}_c$ to $\mathcal{C}_f$ induces a functor $Q': \pi \mathcal{C}_f \rightarrow \pi \mathcal{C}_{cf}$.  The  restriction  of  the  functor $R:\mathcal{C} \rightarrow \pi \mathcal{C}_f$ to $\mathcal{C}_c$ induces a functor $R':\pi \mathcal{C}_c \rightarrow \pi \mathcal{C}_{cf}$.
\end{lema}

\begin{proof}
([3], p. 97).
\end{proof}

\begin{dfn}
The homotopy category $\mathrm{Ho}(\mathcal{C})$ of a model category $\mathcal{C}$ is the category with the same objects as $\mathcal{C}$ and with
$$\mathrm{Hom}_{\mathrm{Ho}(\mathcal{C})}(X,Y) = \mathrm{Hom}_{\pi\mathcal{C}_{cf}}(R'QX,R'QY)= \pi (RQX,RQY).$$
\end{dfn}

As pointed out in ([3], 5.7. Remark), there is a functor $\gamma:\mathcal{C} \rightarrow \mathrm{Ho}(\mathcal{C})$ which is the identity on objects and sends a  map $f:X \rightarrow Y$ to  the  map $R'Q(f):R'Q(X) \rightarrow R'Q(Y)$.  If  each of  the  objects $X$ and $Y$ is  both   fibrant  and  cofibrant,  then  by  construction  the map $\gamma:\mathrm{Hom}_{\mathcal{C}}(X,Y) \rightarrow \mathrm{Hom}_{\mathrm{Ho}(\mathcal{C})}(X,Y)$  is  surjective  and  induces  a  bijection $\pi(X,Y) \cong \mathrm{Hom}_{\mathrm{Ho}(\mathcal{C})}(X,Y)$.

\begin{pps}
If $f$ is a morphism of $\mathcal{C}$,then $\gamma(f)$ is an isomorphism in $\mathrm{Ho}(\mathcal{C})$ if and only if $f$ is a weak equivalence. The morphisms of $\mathrm{Ho}(\mathcal{C})$ are generated under composition by the images under $\gamma$ of morphisms of $\mathcal{C}$ and the inverses of images under $\gamma$ of weak equivalences in $\mathcal{C}$.
\end{pps}

\begin{proof}
([3], pp. 97-98).
\end{proof}

\begin{cor}
If $F$ and $G$ are two functors $\mathrm{Ho}(\mathcal{C}) \rightarrow \mathcal{D}$ and $t:F \gamma \rightarrow G \gamma$ is a natural transformation, thent also gives a natural transformation from $F$ to $G$.
\end{cor}

\begin{proof}
([3], p. 27).
\end{proof}

\begin{lema}
Let $\mathcal{C}$ be a model category and $F:\mathcal{C} \rightarrow \mathcal{D}$ be a functor taking weak equivalences in $\mathcal{C}$ into isomorphisms in $\mathcal{D}$. If $f \sim_l g:A \rightarrow X$ or $f \sim_r g:A \rightarrow X$, then $F(f) = F(g)$ in $\mathcal{D}$.
\end{lema}

\begin{proof}
([3], p. 98).
\end{proof}

\begin{pps}
Suppose  that $A$ is  a  cofibrant  object  of $\mathcal{C}$ and $X$ is  a   fibrant object of $\mathcal{C}$.  Then  the  map $\gamma:\mathrm{Hom}_{\mathcal{C}}(A,X) \rightarrow \mathrm{Hom}_{\mathrm{Ho}(\mathcal{C})}(A,X)$ is  surjective,  and induces a bijection $\pi (A,X) \cong \mathrm{Hom}_{\mathrm{Ho}(\mathcal{C})}(A,X)$.
\end{pps}

\begin{proof}
([3], pp. 98-99).
\end{proof}

\begin{cor}
The canonical functor $\mathcal{C}_{cf} \rightarrow \mathcal{C}_{cf} / \sim$ has the same universal property as the functor $\mathcal{C}_{cf} \rightarrow \mathrm{Ho} \mathcal{C}_{cf} = \mathcal{W}^{-1} \mathcal{C}_{cf}$. Thus, there is an isomorphism of categories $\mathcal{C}_{cf} / \sim \rightarrow \mathrm{Ho} \mathcal{C}_{cf}$. In particular, $\mathrm{Ho} \mathcal{C}_{cf}$ is small.
\end{cor}

\subsection{Weak Factorization Systems}

\begin{dfn}
In a category $\mathcal{C}$, we say that the morphism $f:A \rightarrow B$ has the left lifting property with respect to the morphism $g:C \rightarrow D$ if for any commutative diagram of solid arrows
$$\xymatrix{& A \ar[d]_{f} \ar[r] & C \ar[d]^g 
		\\ & B \ar[r] \ar@{-->}[ur]^h & D}$$
\noindent there is a morphism $h$ which makes the complete diagram commutative. We will write $f \boxslash g$ if $f$ has the left lifting property with respect to $g$.
\noindent For any class of morphisms $S$, we define
$$S^{\boxslash} = \{g \in \mathcal{C} \mid f \boxslash g \textrm{ para todo } f \in S\},$$
	
	$$^{\boxslash}S = \{f \in \mathcal{C} \mid f \boxslash g \textrm{ para todo } g \in S\}.$$
	\noindent Note that for any set $S$, the sets $S^{\boxslash}$ and $^{\boxslash}S$ are closed under retracts.
\end{dfn}

\begin{dfn}
A maximal lifting system $(\mathcal{L},\mathcal{R})$ in a category $\mathcal{C}$ is a pair of classes of morphisms, such that $\mathcal{L} = {^{\boxslash}\mathcal{R}}$ and $\mathcal{R} = \mathcal{L}^{\boxslash}$.
\end{dfn}

The following theorem is well-known; for a proof (and a more general statement), see ([19], 14.1.8).

\begin{thm}[Folklore]
    If $(\mathcal{L},\mathcal{R})$ is a maximal lifting system in a category $\mathcal{C}$, $\mathcal{L}$ and $\mathcal{R}$ contain all isomorphisms and are closed under composition and retraction. Moreover, $\mathcal{L}$ is closed under coproducts and pushouts along morphisms in $\mathcal{C}$, and $\mathcal{R}$ is closed under products and pullbacks along morphisms in $\mathcal{C}$.
\end{thm}

\begin{dfn}
A weak factorization system $(\mathcal{L},\mathcal{R})$ in the category $\mathcal{C}$ is a maximal lifting system such that any morphism in $\mathcal{C}$ can be factored as $g \circ f$ with $f \in \mathcal{L}$ and $g \in \mathcal{R}$.
\end{dfn}

The following is a well-known result for recognizing weak factorization systems (WFSs); for a proof, see ([19], 14.1.13).

\begin{lema}[Folklore]
If $(\mathcal{L},\mathcal{R})$ is a pair of classes of morphisms in a category $\mathcal{C}$ such that
\begin{enumerate}
    \item $f \boxslash g$ for all $f \in \mathcal{L}$ and $g \in \mathcal{R}$,
    \item all morphisms $f \in \mathcal{C}$ can be factored as $f_R \circ f_L$, where $f_R \in \mathcal{R}$ and $f_L \in \mathcal{L}$, and
    \item $\mathcal{L}$ and $\mathcal{R}$ are closed under retracts,
\end{enumerate}
\noindent then $(\mathcal{L},\mathcal{R})$ is a WFS.
\end{lema}

As an example of how lifting properties can classify properties of morphisms, we present the following characterization of retractions and sections.

\begin{dfn}
A morphism $r: A \rightarrow B$ in a category is called a retraction if it is possible to factorize the identity of $B$ as $\mathrm{id}_B = r \circ s$ for some morphism $s$. Dually, a morphism $s:A \rightarrow B$ is called a section if it is possible to factorize the identity of $A$ as $\mathrm{id}_A = r \circ s$ for some morphism $r$.
\end{dfn}

\begin{lema}
The class of retractions is exactly $\{\emptyset \rightarrow A \mid A \in \mathcal{C}\}^{\boxslash}$. Dually, the class of sections is exactly $^{\boxslash}\{A \rightarrow * \mid A \in \mathcal{C}\}$.
\end{lema}

\section{Retracts and Cores}

In this section, we introduce the notion of \emph{cores}.

\begin{dfn}
Let $\mathfrak{A}$ be a $\sigma$-structure. An endomorphism $f:\mathfrak{A} \rightarrow \mathfrak{A}$ is a retraction if it leaves its image fixed, in other words if $f(x)=x$ for all $x \in f[A]$. A substructure $\mathfrak{B}$ of $\mathfrak{A}$ is called a retract of $\mathfrak{A}$ if there exists a retraction of $\mathfrak{A}$ onto $\mathfrak{B}$; a retract is proper if it is a proper substructure.
\end{dfn}

\begin{lema}
If $\mathfrak{B}$ is a retract of $\mathfrak{A}$, then $\mathfrak{A}$ and $\mathfrak{B}$ are homomorphically equivalent.
\end{lema}

\begin{proof}
([11], p. 12).
\end{proof}

\begin{dfn}
A $\sigma$-structure $\mathfrak{C}$ is called a \emph{core} if it has no proper retracts. A retract $\mathfrak{C}$ of $\mathfrak{A}$ is called \emph{a core of $\mathfrak{A}$} if it is a core.
\end{dfn}

\begin{lema}[Characterisation of cores]
For a $\sigma$-structure $\mathfrak{C}$ the following conditions are equivalent.
\begin{enumerate}
    \item $\mathfrak{C}$ is a core (that is, $\mathfrak{C}$ has no proper retracts).
    \item $\mathfrak{C}$ is not homomorphic to any proper substructure of $\mathfrak{C}$.
    \item Every endomorphism of $\mathfrak{C}$ is an automorphism.
\end{enumerate}
\end{lema}

\begin{proof}
([11], p. 11).
\end{proof}

\begin{lema}
Let $\mathfrak{A}$ and $\mathfrak{B}$ be two $\sigma$-structures. If there exist surjective homomorphisms $f:\mathfrak{A} \rightarrow \mathfrak{B}$ and $g:\mathfrak{B} \rightarrow \mathfrak{A}$, then $\mathfrak{A}$ and $\mathfrak{B}$ are isomorphic.
\end{lema}

\begin{proof}
([11], p. 12).
\end{proof}

\begin{lema}
Let $\mathfrak{C}$ and $\mathfrak{C}'$ be two cores. If $\mathfrak{C}$ and $\mathfrak{C}'$ are homomorphically equivalent, they are isomorphic.
\end{lema}

\begin{proof}
([11], p. 12).
\end{proof}

\begin{pps}
Every $\sigma$-structure $\mathfrak{A}$ has a unique core $\mathfrak{C}$ (up to isomorphism). Moreover, $\mathfrak{C}$ is the unique core to which $\mathfrak{A}$ is homomorphically equivalent.
\end{pps}

\begin{proof}
([11], p. 12).
\end{proof}

\begin{cor}
A $\sigma$-structure $\mathfrak{C}$ is a core if and only if it is not homomorphically equivalent to a $\sigma$-structure with fewer vertices.
\end{cor}

\begin{proof}
([11], p. 12).
\end{proof}

\subsection{Structures and Homomorphisms over a Set X}

Now we will see one more characterization of cores, namely, when referring to a given subset $X$ of the universe $A$ of a given structure $\mathfrak{A}$. This characterization will be important when dealing with definitions of $k$-homomorphisms and $k$-cores. Here I will follow [21].

\begin{dfn}
Let $X$ be an arbitrary set. We call a structure $\mathfrak{A}$ whose universe includes $X$ a \emph{structure over $X$}. For structures $\mathfrak{A}$ and $\mathfrak{B}$ over $X$, we call a homomorphism from $\mathfrak{A}$ to $\mathfrak{B}$ which fixes $X$ pointwise a \emph{homomorphism over $X$}. We write $\mathfrak{A} \rightarrow_X \mathfrak{B}$ if there exists a homomorphism
from $\mathfrak{A}$ to $\mathfrak{B}$ over $X$. We say $\mathfrak{A}$ and $\mathfrak{B}$ are homomorphically equivalent over $X$, and we write
$\mathfrak{A} \rightleftarrows_X \mathfrak{B}$, if $\mathfrak{A} \rightarrow_X \mathfrak{B}$ and $\mathfrak{B} \rightarrow_X \mathfrak{A}$. We say $\mathfrak{A}$ and $\mathfrak{B}$ are isomorphic over $X$, and we write
$\mathfrak{A} \cong_X \mathfrak{B}$, if there exist homomorphisms $f:\mathfrak{A} \rightarrow_X \mathfrak{B}$ and $g:\mathfrak{B} \rightarrow_X \mathfrak{A}$ such that $g \circ f = \mathrm{id}_A$
and $f \circ g = \mathrm{id}_B$; in this case, we say $f$ and $g$ are \emph{isomorphisms over $X$}.
\end{dfn}

By default, graphs are simple (i.e., undirected and without self-loops).


For a subset $X \subseteq A$, let $\mathcal{G}(\mathfrak{A}) \backslash X$ denote the induced subgraph of $\mathcal{G}(\mathfrak{A})$ with vertex set $A \backslash X$.

The tree-depth $\mathrm{td}_X(\mathfrak{A})$ of a finite structure $\mathfrak{A}$ over a subset $X \subseteq A$ is defined as the tree-depth of the Gaifman graph of $\mathfrak{A}$ over $X$: $\mathrm{td}_X(\mathfrak{A}) = \mathrm{td}(\mathcal{G}(\mathfrak{A}) \backslash X)$.

Suppose $\mathfrak{B}$ is a substructure of $\mathfrak{A}$. Homomorphisms $\mathfrak{A} \rightarrow_B \mathfrak{B}$ are called retractions.

A structure $\mathfrak{A}$ is a core over a subset $X \subseteq A$ if every homomorphism $\mathfrak{A} \rightarrow_X \mathfrak{A}$ is an automorphism.

\begin{lema}
Let $\mathfrak{A}$ be a finite structure and let $X \subseteq A$.
\begin{enumerate}
    \item $\mathfrak{A}$ is a core over $X$ if, and only if, it has no proper retract over X. (i.e., $\mathfrak{A}$ is a retract of $\mathfrak{B} \Rightarrow A = B$ or $X \not\subseteq B$).
    \item $\mathfrak{A}$ has a retract which is a core over $X$. Moreover, if $\mathfrak{A}$ is a retract of $\mathfrak{B}_1$ and $\mathfrak{A}$ is a retract of $\mathfrak{B}_2$ such that both $\mathfrak{B}_1$ and $\mathfrak{B}_2$ are cores over $X$, then $\mathfrak{B}_1 \cong_X \mathfrak{B}_2$.
\end{enumerate}
\end{lema}

\begin{dfn}
For every finite set $X$, we fix some set $\mathscr{C}_X$ of finite cores over $X$ containing exactly one representative from every $\cong_X$-equivalence class of finite structures. Since $\mathscr{C}_X$ contains only finite structures, every which is unique up to isomorphism over $X$, it follows that that $\mathscr{C}_X$ is a countably infinite set. We will call members of $\mathscr{C}_X$ canonical cores over $X$.
\end{dfn}

\begin{cor}
For every finite structure $\mathfrak{A}$ and $X \subseteq A$, there exists a unique $\mathfrak{C} \in \mathscr{C}_X$ such that $\mathfrak{A} \rightleftarrows_X \mathfrak{C}$. Moreover, $\mathrm{td}_X(\mathfrak{C}) \leq \mathrm{td}_X(\mathfrak{A})$ and every homomorphism $h : \mathfrak{C} \rightarrow_X \mathfrak{A}$ is injective and has the property that $\mathfrak{A}$ is a retract of $h(\mathfrak{C})$. 
\end{cor}

We call $\mathfrak{C}$ the (canonical) core of $\mathfrak{A}$ over $X$ and denoted it by $\mathrm{\mathbf{Core}}_X(\mathfrak{A})$. For the special case where $X = \emptyset$, we write $\mathscr{C}$ instead of $\mathscr{C}_{\emptyset}$; and $\mathrm{\mathbf{Core}}(\mathfrak{A})$ instead of $\mathbf{Core}_{\emptyset}(\mathfrak{A})$.

\subsection{k-Homomorphisms and k-Cores}

\begin{dfn}
Let $k \in \mathbb{N}$. We write $\mathfrak{A} \rightarrow_X^n \mathfrak{B}$ and say $\mathfrak{A}$ is $k$-homomorphic to $\mathfrak{B}$ over $X$
if $(\mathfrak{C} \rightarrow_X \mathfrak{A}) \Rightarrow (\mathfrak{C} \rightarrow_X \mathfrak{B})$ for every finite structure $\mathfrak{C}$ of tree-depth at most $k$ over $X$. We write $\mathfrak{A} \rightleftarrows_X^n \mathfrak{B}$ and say $\mathfrak{A}$ and $\mathfrak{B}$ are $k$-homomorphically equivalent over $X$ if $\mathfrak{A} \rightarrow_X^k \mathfrak{B}$ and $\mathfrak{B} \rightarrow_X^k \mathfrak{A}$. As usual, we write $\mathfrak{A} \rightarrow^k \mathfrak{B}$ (resp. $\mathfrak{A} \rightleftarrows^k \mathfrak{B}$) if $\mathfrak{A} \rightarrow_{\emptyset}^k \mathfrak{B}$ (resp. $\mathfrak{A} \rightleftarrows_{\emptyset}^k \mathfrak{B}$.
\end{dfn}

For the next few definitions, let $X$ be a fixed finite set.

\begin{dfn}
For $k \in \mathbb{N}$, let $\mathscr{C}_X^k$ denote the set of finite canonical cores over $X$ with tree-depth at most $k$ over $X$. That is, $\mathscr{C}_X^k = \{\mathfrak{C} \in \mathscr{C}_X : \mathrm{td}_X(\mathfrak{C}) \leq k \}$. Members of $\mathscr{C}_X^k$ are called $k$-cores over $X$.
\end{dfn}

Now, some obvious properties:

\begin{itemize}
	
	\item $\mathscr{C}_X = \bigcup_{k \in \mathbb{N}} \mathscr{C}_X^k$.
	
	\item $\mathfrak{A} \rightarrow_X^k \mathfrak{B}$ if, and only if, $\mathfrak{C} \rightarrow_X \mathfrak{A} \Leftrightarrow \mathfrak{C} \rightarrow_X \mathfrak{B}$, for every $\mathfrak{C} \in \mathscr{C}_X^k$.
	
	\item $\mathfrak{C}_1 \rightarrow_X \mathfrak{C}_2$ if, and only if, $\mathfrak{C}_1 \rightarrow_X^k \mathfrak{C}_2$, for every $\mathfrak{C}_1, \mathfrak{C}_2 \in \mathscr{C}_X^k$; that is, $\rightarrow_X$ coincides with $\rightarrow_X^k$ on the class $\mathscr{C}_X^k$.
	
\end{itemize}

Fourth, $\rightarrow_X$ partially orders $\mathscr{C}_X^k$. This is obvious, since $\mathscr{C}_X^k$ is a subset of the homomorphism lattice $(\mathscr{C}_X,\rightarrow_X)$.

\begin{lema}
$(\mathscr{C}_X^k,\rightarrow_X)$ is an upper semilattice. That is, every two structures in $\mathscr{C}_X^k$ have a least upper bound (l.u.b.) with respect to $\rightarrow_X$. Moreover, the l.u.b. of two structures in $(\mathscr{C}_X^k,\rightarrow_X)$ coincides with their l.u.b. in the lattice $(\mathscr{C}_X,\rightarrow_X)$.
\end{lema}

\begin{proof}
([21], p. 22)
\end{proof}

\begin{pps}
Up to $\rightleftarrows_X$, there are only finitely many finite structures over $X$ with tree-depth  $\leq n$ over $X$. Equivalently, there are only finitely many canonical cores over $X$ of tree-depth $\leq n$ over X (i.e., $\mathscr{C}_X^k$  is a finite set).
\end{pps}

\begin{proof}
([21], pp. 22-23).
\end{proof}{}

\begin{dfn}[$k$-Core]
For a structure $\mathfrak{A}$ and a finite set $X \subseteq A$, the $k$-core $\mathrm{\mathbf{Core}}_X^k(\mathfrak{A})$ of $\mathfrak{A}$ over $X$ is the least upper bound of $\{\mathfrak{C} \in \mathscr{C}_X^k \mid \mathfrak{C} \rightarrow_X \mathfrak{A}\}$ in the complete upper semilattice $(\mathscr{C}_X^k,\rightarrow_X)$.
\end{dfn}

\begin{lema}
$\mathfrak{A} \rightarrow_X^k \mathfrak{B}$ if, and only if, $\mathrm{\mathbf{Core}}_X^k(\mathfrak{A}) \rightarrow_X \mathfrak{B}$.
\end{lema}{}

\begin{proof}
([21], p. 23).
\end{proof}

\subsection{Logical Characterization of k-Homomorphism}

Recall that existential-positive formulas are built out of atomic formulas using only conjunction, disjunction and existential quantification. Primitive-positive formulas are precisely the existential-positive formulas containing no disjunctions.

\begin{lema}
$\mathfrak{A} \rightarrow^n \mathfrak{B}$ if, and only if, $\mathfrak{A} \models \theta \Rightarrow \mathfrak{B} \models \theta$, para cada sentença primitiva-positiva $\theta$ of quantifier-rank $k$.  
\end{lema}

\begin{proof}
([21], p. 24).
\end{proof}

\section{The Category of Neighborhoods}

\begin{dfn}
To a vocabulary $\sigma$ we have a set $\mathcal{T}(\sigma)$ containing the closed $\sigma$-terms. $\mathcal{T}$ is given by recursion:
\begin{itemize}
    \item $\mathcal{T}(\sigma)$ contains all constant symbols;
    \item if $f_i$ is an $k$-ary function symbol of $\sigma$, and $t_1,...,t_k \in \mathcal{T}(\sigma)$, then $f(t_1,...,t_k) \in \mathcal{T}(\sigma)$.
\end{itemize}
\end{dfn}{}

\begin{rem}
The above definition of $\mathcal{T}(\sigma)$ has two subtle issues. For one we did not specify exactly what a term is. Secondly it is not clear that the above recursive definition actually defines a set. To actually justify these details requires quite a bit of set theory.
\end{rem}

\begin{dfn}
For $\sigma$ a vocabulary we define the \emph{free term $\sigma$-structure} $\widetilde{\mathcal{T}(\sigma)}$ to be the $\sigma$-structure with domain $\mathcal{T}(\sigma)$ and with interpretations as follows:
\begin{itemize}
    \item For every constant symbol $c$, we set $c^{\widetilde{\mathcal{T}(\sigma)}} = c$.
    \item For every $k$-ary function symbol $f_i$, with $a_1,...,a_n \in \mathcal{T}(\sigma)$, we set
    $$f_i^{\widetilde{\mathcal{T}(\sigma)}} = f_i(a_1,...,a_n).$$
    \item For every $k$-ary relation symbol $R_i$, we let $R_i^{\widetilde{\mathcal{T}(\sigma)}} = \emptyset$.
\end{itemize}
\end{dfn}

Furthermore $\widetilde{\mathcal{T}(\sigma)}$ has a universal property.

\begin{pps}
    For any $\sigma$-structure $\mathfrak{A}$ then there exists a unique homomorphism of $\sigma$-structures $\rho: \widetilde{\mathcal{T}(\sigma)} \rightarrow \mathfrak{A}$. 
\end{pps}

\begin{proof}
The map $\rho: \widetilde{\mathcal{T}(\sigma)} \rightarrow \mathfrak{A}$  is defined as follows:
\begin{itemize}
    \item for a constant symbol $c$, $\rho(c^{\widetilde{\mathcal{T}(\sigma)}}) = c^{\mathfrak{A}}$; and 
    \item if $t \in \mathcal{T}(\sigma)$ has the form $f(t_1,...,t_n)$ then $\rho(t) = f^{\mathfrak{A}}(\rho(t_1),...,\rho(t_n))$.
\end{itemize}
This is well-defined since we have a unique parsing lemma for terms. Furthermore $\rho$ is clearly a homomorphism. On constant
and function symbols it is defined as is should be and for relation symbols the
claim is vacuous since $R^{\widetilde{\mathcal{T}(\sigma)}} = \emptyset$.
For the uniqueness we use induction on the complexity of terms. Suppose $\rho,\xi: \widetilde{\mathcal{T}(\sigma)} \rightarrow \mathfrak{A}$ are homomorphisms. Then
\begin{itemize}
   
       \item for every constant symbol $c$, $\rho(c^{\widetilde{\mathcal{T}(\sigma)}}) = c^{\mathfrak{A}} = \xi(c)$;
       
       \item if $t \in \mathcal{T}(\sigma)$ has the form $f(t_1,...,t_n)$, then $\rho(t) = f^{\mathfrak{A}}(\rho(t_1),...,\rho(t_n)) = f^{\mathfrak{A}}(\xi(t_1),...,\xi(t_n)) = \xi(t)$, since the $t_i$'s have lower complexity than $t$.
       
       \noindent Thus $\rho = \xi$.
   \end{itemize}
\end{proof}

\begin{dfn}
Let $T=\{1\}$ and let $\top$ be the $\sigma$-structure such that:
\begin{itemize}
    \item the domain of $\top$ is $T$;
    \item for every $k$-ary relation symbol $R_i$, $R_i^{\top} = T^k$;
    \item for every constant symbol $c$, $c^{\top} = 1$.
\end{itemize}{}
\end{dfn}{}

There exists exactly one homomorphism from any $\sigma$-structure to $\top$, namely the constant mapping to 1.

\begin{dfn}
Let $N_{d}^{\mathfrak{A}}(\vec{a})$ and $N_{d}^{\mathfrak{B}}(\vec{b})$ be $\sigma_n$-structures. A homomorphism $h_{d}^{\sigma_n}: N_{d}^{\mathfrak{A}}(\vec{a}) \rightarrow N_{d}^{\mathfrak{B}}(\vec{b})$ is defined as the
homomorphism $h: \mathfrak{A} \rightarrow \mathfrak{B}$ such that the function $h: A \rightarrow B$, between the universes $A$ and $B$ of $\mathfrak{A}$ and $\mathfrak{B}$, respectively, is restricted to balls $B_{d}^{\mathfrak{A}}(a)$ and $B_{d}^{\mathfrak{B}}(b)$, that is, it is
a function $h_{d}^{\sigma_n}: B_{d}^{\mathfrak{A}}(\vec{a}) \rightarrow B_{d}^{\mathfrak{B}}(\vec{b})$ such that:
\begin{enumerate}
    \item For each $k$-ary relation symbol $R_i$, interpreted as $R_i^{\mathfrak{A}}$ restricted to $B_{d}^{\mathfrak{A}}(\vec{a})$, that is, $R_i^{\mathfrak{A}} \cap (B_{d}^{\mathfrak{A}}(\vec{a}))^k$, and a tuple $(x_1^d,...,x_k^d) \in R_i^{\mathfrak{A}} \cap (B_{d}^{\mathfrak{A}}(\vec{a}))^k$, the tuple $(h^{\sigma_n}(x_1^d),...,h^{\sigma_n}(x_k^d))$ is in $R_i^{\mathfrak{B}} \cap (B_{d}^{\mathfrak{B}}(\vec{b}))^k$; and
    \item The constant symbols in $\sigma_n$ that are interpreted as $(a_1,...,a_n) = \vec{a}$ in $\mathfrak{A}$ are interpreted as $(h(a_1),...,h(a_n)) = \vec{b}$ in $\mathfrak{B}$.
\end{enumerate}
\end{dfn}

\begin{dfn}
The Category $\mathbf{STRUCT}[\sigma_n]_{(d,0)}^{\widetilde{\mathcal{T}(\sigma_n)}}$ is defined as follow:
\begin{itemize}
    \item objects: $d$-neighborhoods $N_d^{\mathfrak{A}}(\vec{a})$, $0$-neighborhood $N_0^{\mathfrak{A}}(a)$ and $\widetilde{\mathcal{T}(\sigma_n)}$;
    \item morphisms: 
    \begin{itemize}
    \item homomorphisms $h_{d}^{\sigma_n}: N_{d}^{\mathfrak{A}}(\vec{a}) \rightarrow N_{d}^{\mathfrak{B}}(\vec{b})$;
        \item the only homomorphism $\rho_{\sigma_n}: \widetilde{\mathcal{T}(\sigma_n)} \rightarrow N_{d}^{\mathfrak{A}}(\vec{a})$, for every $\sigma_n$-structure $N_{d}^{\mathfrak{A}}(\vec{a})$ of $\mathbf{STRUCT}[\sigma_n]_{d}^{\widetilde{\mathcal{T}(\sigma_n)}}$; and
    \item the only homomorphism $T: N_{d}^{\mathfrak{A}}(\vec{a}) \rightarrow N_0^{\mathfrak{A}}(a)$, for every $\sigma_n$-structure $N_{d}^{\mathfrak{A}}(\vec{a})$ of $\mathbf{STRUCT}[\sigma_n]_{d}^{\widetilde{\mathcal{T}(\sigma_n)}}$.
    \end{itemize}{}
\end{itemize}{}
\end{dfn}{}

Remember that a category $\mathcal{C}$ is finitely complete if it has a terminal object and admits all binary products and equalizers; dually, $\mathcal{C}$ is finitely cocomplete if it has a initial object and admits all binary coproducts and coequalizers.

\begin{pps}
    The category $\mathbf{STRUCT}[\sigma_n]_{(d,0)}^{\widetilde{\mathcal{T}(\sigma_n)}}$ is finitely complete and finitely cocomplete.
\end{pps}{}

\begin{proof}
The terminal object is $N_0^{\mathfrak{A}}(a)$, the initial object is $\widetilde{\mathcal{T}(\sigma_n)}$.

For two $\sigma_n$-finite relational structures $N_{d}^{\mathfrak{A}}(\vec{a})$ and $N_{d}^{\mathfrak{B}}(\vec{b})$, $N_{d}^{\mathfrak{A}}(\vec{a}) \times N_{d}^{\mathfrak{B}}(\vec{b})$ is the $\sigma_n$-structure defined on the Cartesian product $B_{d}^{\mathfrak{A}}(a) \times B_{d}^{\mathfrak{B}}(b)$.

For two $\sigma_n$-finite relational structures $N_{d}^{\mathfrak{A}}(\vec{a})$ and $N_{d}^{\mathfrak{B}}(\vec{b})$, $N_{d}^{\mathfrak{A}}(\vec{a}) \coprod N_{d}^{\mathfrak{B}}(\vec{b})$ is the $\sigma_n$-structure defined on the disjoint union $B_{d}^{\mathfrak{A}}(a) \coprod B_{d}^{\mathfrak{B}}(b)$.

For two morphisms $f,g: N_{d}^{\mathfrak{A}}(\vec{a}) \rightarrow N_{d}^{\mathfrak{B}}(\vec{b})$, the equalizer of $f$ and $g$ is the substructure induced by the ''vertex'' sent by $f$ and $g$ for the same ''vertex''; that is, it is the substructure induced by the set $\{x \in B_{d}^{\mathfrak{A}}(a) \mid f(x) = g(x) \} \subseteq B_{d}^{\mathfrak{A}}(a)$.

For two morphisms $f,g: N_{d}^{\mathfrak{A}}(\vec{a}) \rightarrow N_{d}^{\mathfrak{B}}(\vec{b})$, the coequalizer of $f$ and $g$ is the substructure induced by the quotient of $N_{d}^{\mathfrak{B}}(\vec{b})$ by the equivalence relation generated by the set of ''vertex'' pairs $\{\langle f(x),g(x) \rangle \mid x \in B_{d}^{\mathfrak{A}}(a)\}$.

\end{proof}{}

\subsection{Homomorphism as (Partial) (Quasi-) Order}

There is a quasi-order in any category $\mathcal{C}$ induced by its arrows. In addition, we can define an associated equivalence relation on $\mathcal{C}$ as follows:

$$A \equiv_{\rightarrow} B \Leftrightarrow A \rightarrow B \wedge B \rightarrow A.$$

Thus, $\rightarrow$ induces a partial order on $\mathcal{C} / \equiv_{\rightarrow}$. For the case of $\mathcal{C} = \mathbf{STRUCT}[\sigma]$, $\mathcal{C} / \equiv_{\rightarrow}$ is a set, and we have the poset

$$\mathscr{C}(\mathbf{STRUCT}[\sigma]) = \langle \mathbf{STRUCT}[\sigma] / \equiv_{\rightarrow}, \rightarrow \rangle.$$ 

Now, notice that there is a canonical functor

$$\mathscr{F}_{\sigma}:\mathbf{STRUCT}[\sigma] \rightarrow \mathbf{STRUCT}[\sigma] / \equiv_{\rightarrow}$$

\noindent sending every $\sigma$-structure $\mathfrak{A}$ to its class of $\equiv_{\rightarrow}$-equivalence $[\mathfrak{A}]$, and every homomorphism of $\sigma$-structures $f: \mathfrak{A} \rightarrow \mathfrak {B}$ for the morphism $[f]: [\mathfrak{A}] \rightarrow [\mathfrak{B}]$.

Note also that the $\equiv_{\rightarrow}$-equivalence classes formed by $\sigma$-cores are, by Lemma 24, $\cong$-equivalence classes. Thus, the set $\mathscr{C}$ of finite $\sigma$-cores is contained in $\mathbf{STRUCT}[\sigma] / \equiv_{\rightarrow}$. In addition, under Corollary 4, every $\equiv_{\rightarrow}$-equivalence class in $\mathbf{STRUCT}[\sigma] / \equiv_{\rightarrow}$ has a single representative, up to isomorphism, in $\mathscr{C}$. Thus, if $f: \mathfrak{A} \rightarrow \mathfrak{B}$ is a homomorphism of $\sigma$-structures, where $\mathfrak{A}$ is $\equiv_{\rightarrow}$-equivalent to $\mathfrak{B}$, then $[f]: [\mathfrak{A}] \rightarrow [\mathfrak{B}]$ is an isomorphism in $\mathbf{STRUCT}[\sigma] / \equiv_{\rightarrow}$.

The same construction works for the category $\mathbf{STRUCT}[\sigma_n]_{(d,0)}^{\widetilde{\mathcal{T}(\sigma_n)}}$. That is, we also have the poset

$$\mathscr{C}(\mathbf{STRUCT}[\sigma_n]_{(d,0)}^{\widetilde{\mathcal{T}(\sigma_n)}}) = \langle \mathbf{STRUCT}[\sigma_n]_{(d,0)}^{\widetilde{\mathcal{T}(\sigma_n)}} / \equiv_{\rightarrow}, \rightarrow \rangle,$$ 

\noindent with a canonical functor

$$\mathscr{F}_{[\sigma_n]_{(d,0)}}:\mathbf{STRUCT}[\sigma_n]_{(d,0)}^{\widetilde{\mathcal{T}(\sigma_n)}} \rightarrow \mathbf{STRUCT}[\sigma_n]_{(d,0)}^{\widetilde{\mathcal{T}(\sigma_n)}} / \equiv_{\rightarrow}$$ 

\noindent sending every $d$-neighborhood $N_{d}^{\mathfrak{A}}(\vec{a})$ of a given $n$-tuple of points $\vec{a}$ in a given $\sigma$-structure $\mathfrak{A}$ for its $\equiv_{\rightarrow}$-equivalence class $[N_{d}^{\mathfrak{A}}(\vec{a})]$, and every homomorphism of $d$-neighborhoods $f_{d}^{\sigma_n}: N_{d}^{\mathfrak{A}}(\vec{a}) \rightarrow N_{d}^{\mathfrak{B}}(\vec{b})$ for the morphism $[f_{d}^{\sigma_n}]: [N_{d}^{\mathfrak{A}}(\vec{a})] \rightarrow [N_{d}^{\mathfrak{B}}(\vec{b})]$.

Similarly, the $\equiv_{\rightarrow}$-equivalence classes formed by $[\sigma_n]_{(d,0)}$-cores are, by Lemma 24, $\cong$-equivalence classes. Thus, we have the set $\mathscr{C}_{[\sigma_n]_{(d,0)}}$ of finite $[\sigma_n]_{(d,0)}$-cores (that is, the cores of $d$-neighborhoods of $n$-tuples of points $\vec{a}$ of $\sigma$-structures) contained in $\mathbf{STRUCT}[\sigma_n]_{(d,0)}^{\widetilde{\mathcal{T}(\sigma_n)}} / \equiv_{\rightarrow}$. In addition, also by Corollary 4, every $\equiv_{\rightarrow}$-equivalence class in $\mathbf{STRUCT}[\sigma_n]_{(d,0)}^{\widetilde{\mathcal{T}(\sigma_n)}} / \equiv_{\rightarrow}$ has a single representative, up to isomorphism, in $\mathscr{C}_{[\ sigma_n]_{(d,0)}}$. Thus, if $f_{d}^{\sigma_n}: N_{d}^{\mathfrak{A}}(\vec{a}) \rightarrow N_{d}^{\mathfrak{B}}(\vec{b})$ is a homomorphism of $d$-neighborhoods, where $N_{d}^{\mathfrak{A}}(\vec{a})$ is $\equiv_{\rightarrow}$-equivalent to $N_{d}^{\mathfrak{B}}(\vec{b})$, then $[f_{d}^{\sigma_n}]: [N_{d}^{\mathfrak{A}}(\vec{a})] \rightarrow [N_{d}^{\mathfrak{B}}(\vec{b})]$ is an isomorphism in $\mathbf{STRUCT}[\sigma_n]_{(d,0)}^{\widetilde{\mathcal{T}(\sigma_n)}} / \equiv_{\rightarrow}$.

\subsection{k-Homomorphism as (Partial) (Quasi-) Order}

Note that, as with $\rightarrow_X$, the relation $\rightarrow_X^k$ is a quasi-order on structures. It is evident that $\mathfrak{A} \rightarrow_X \mathfrak{B}$ implies $\mathfrak{A} \rightarrow_X^k \mathfrak{B}$, for every $k$. Also, note that $\mathfrak{A} \rightarrow_X^k \mathfrak{B}$, for every $k$ implies $\mathfrak{A} \rightarrow_{X'}^{k'} \mathfrak{B}$, for all $k' \leq k$ and $X' \subseteq X$.

The same process performed in \S 4.1 can be performed here. To do this, let $\mathbf{STRUCT}[\sigma_n]_{(d,0)}^{(\widetilde{\mathcal{T}(\sigma_n)})_k}$ be a subcategory $\mathbf{STRUCT}[\sigma_n]_{(d,0)}^{\widetilde{\mathcal{T}(\sigma_n)}}$ whose objects are the $\sigma_n$-structures $N_{d}^{\mathfrak{A}}(\vec{a})$ such that $N_{d}^{\mathfrak{A}}(\vec{a}) \rightarrow N_{d}^{\mathfrak{D}}(\vec{d})$, where $N_{d}^{\mathfrak{D}}(\vec{d})$ is a $\sigma_n$-structure with tree-depth at most $k$. Thus, we can define on $\mathbf{STRUCT}[\sigma_n]_{(d,0)}^{(\widetilde{\mathcal{T}(\sigma_n)})_k}$ the following equivalence relation:

$$N_{d}^{\mathfrak{A}}(\vec{a}) \equiv_{\rightarrow^k} N_{d}^{\mathfrak{B}}(\vec{b}) \Leftrightarrow N_{d}^{\mathfrak{A}}(\vec{a}) \rightarrow^k N_{d}^{\mathfrak{B}}(\vec{b}) \wedge N_{d}^{\mathfrak{B}}(\vec{b}) \rightarrow^k N_{d}^{\mathfrak{A}}(\vec{a}).$$

Therefore, $\rightarrow^k$ induces a partial order on $\mathbf{STRUCT}[\sigma_n]_{(d,0)}^{(\widetilde{\mathcal{T}(\sigma_n)})_k}$, and $\mathbf{STRUCT}[\sigma_n]_{(d,0)}^{(\widetilde{\mathcal{T}(\sigma_n)})_k} / \equiv_{\rightarrow^k}$ is the poset

$$\mathscr{C}(\mathbf{STRUCT}[\sigma_n]_{(d,0)}^{(\widetilde{\mathcal{T}(\sigma_n)})})^k = \langle \mathbf{STRUCT}[\sigma_n]_{(d,0)}^{(\widetilde{\mathcal{T}(\sigma_n)})_k} / \equiv_{\rightarrow^k}, \rightarrow^k \rangle.$$

As in \S4.1, there is a canonical functor

$$\mathscr{F}^k_{[\sigma_n]_{(d,0)}}:\mathbf{STRUCT}[\sigma_n]_{(d,0)}^{(\widetilde{\mathcal{T}(\sigma_n)})_k} \rightarrow \mathbf{STRUCT}[\sigma_n]_{(d,0)}^{(\widetilde{\mathcal{T}(\sigma_n)})_k} / \equiv_{\rightarrow^k}$$

\noindent sending every $\sigma_n$-structure $N_{d}^{\mathfrak{A}}(\vec{a})$ in $\mathbf{STRUCT}[\sigma_n]_{d}^{(\widetilde{\mathcal{T}(\sigma_n)})_k}$ for its $\equiv_{\rightarrow^k}$-equivalence class $[N_{d}^{\mathfrak{A}}(\vec{a})]^k$, and every homomorphism of $\sigma_n$-structures $f: N_{d}^{\mathfrak{A}}(\vec{a}) \rightarrow N_{d}^{\mathfrak{B}}(\vec{b})$ in $\mathbf{STRUCT}[\sigma_n]_{d}^{(\widetilde{\mathcal{T}(\sigma_n)})_k}$ for the morphism $[f]^k: [N_{d}^{\mathfrak{A}}(\vec{a})]^k \rightarrow [N_{d}^{\mathfrak{B}}(\vec{b})]^k$. 

Similarly, the $\equiv_{\rightarrow^k}$-equivalence classes formed by $\sigma_n$-$k$-cores are, by Lemma 24, $\cong$-equivalence classes. Thus, the set $\mathscr{C}^k$ of finite $\sigma_n$-$k$-cores is contained in $\mathbf{STRUCT}[\sigma_n]_{(d,0)}^{(\widetilde{\mathcal{T}(\sigma_n)})_k}/\equiv_{\rightarrow^k}$. And, again, by Corollary 4, every $\equiv_{\rightarrow^k}$-equivalence class in $\mathbf{STRUCT}[\sigma_n]_{(d,0)}^{(\widetilde{\mathcal{T}(\sigma_n)})_k}/\equiv_{\rightarrow^k}$ has a single representative, up to isomorphism, in $\mathscr{C}^k$. Thus, if $f: N_{d}^{\mathfrak{A}}(\vec{a}) \rightarrow N_{d}^{\mathfrak{B}}(\vec{b})$ is a homomorphism of $\sigma_n$-structures in $\mathbf{STRUCT}[\sigma_n]_{d}^{(\widetilde{\mathcal{T}(\sigma_n)})_k}$, where $N_{d}^{\mathfrak{A}}(\vec{a})$ is $\equiv_{\rightarrow^k}$-equivalent to $N_{d}^{\mathfrak{B}}(\vec{b})$, then $[f]^k: [N_{d}^{\mathfrak{A}}(\vec{a})]^k \rightarrow [N_{d}^{\mathfrak{B}}(\vec{b})]^k$ is an isomorphism in $\mathbf{STRUCT}[\sigma_n]_{(d,0)}^{(\widetilde{\mathcal{T}(\sigma_n)})_k}/\equiv_{\rightarrow^k}$. 

\section{A Homotopic Variation}
\label{sec:others}

\subsection{Generalized Core Model Structure} 

In [4], Droz defines a core-based model structure (which is called a core model structure) on a particular category of graphs, and generalizes it in [5] for any finitely complete and finitely cocomplete category.

\begin{dfn}[DROZ \& ZAKHAREVICH]
	Let $\mathcal{C}$ be a category. We define the preorder $P(\mathcal{C})$ with $\mathrm{ob}P(\mathcal{C}) = \mathrm{ob}\mathcal{C}$, and $\mathrm{Hom}_{P(\mathcal{C})}(X,Y)$ equaling the one-point set if there exists a morphism $X \rightarrow Y \in \mathcal{C}$, and the empty set otherwise. We will write $X \sim_{P(\mathcal{C})} Y$ if $X$ is isomorphic to $Y$ in $P(\mathcal{C})$. 
\end{dfn}

There is a canonical functor $R_{\mathcal{C}}: \mathcal{C} \rightarrow P(\mathcal{C})$, such that any functor $F:\mathcal{C} \rightarrow \mathcal{D}$, where $\mathcal{D}$ is a preorder, factors through $R_{\mathcal{C}}$. Droz and Zakharevich then define a model structure on $\mathcal{C}$ such
that the weak equivalences are $R_{\mathcal{C}}^{-1}(\mathrm{iso}P(\mathcal{C}))$.

Droz and Zakharevich then prove the following result:

\begin{thm}[Generalized core model structure on $\mathcal{C}$]
    There is a model structure $\mathbb{C}^{\mathrm{core}}$ with homotopy category $P(\mathcal{C})$ on any bicomplete category $\mathcal{C}$. A morphism $f: A \rightarrow B$ is a weak equivalence iff $A \sim_{P(\mathcal{C})} B$. The acyclic fibrations are exactly the retractions in $\mathcal{C}$.
\end{thm}

\begin{proof}
([5], pp. 29-30).
\end{proof}{}


\subsection{Generalized Core Model Structure on Category of Neighborhoods}

Since $\mathbf{STRUCT}[\sigma_n]_{(d,0)}^{\widetilde{\mathcal{T}(\sigma_n)}}$ is finitely complete and finitely complete, the application of the Theorem 2 is immediate, and we have a generalized core model structure on $\mathbf{STRUCT}[\sigma_n]_{(d,0)}^{\widetilde{\mathcal{T}(\sigma_n)}}$ with the following characteristics.

First, the universal funtor of $\mathbf{STRUCT}[\sigma_n]_{(d,0)}^{{\widetilde{\mathcal{T}(\sigma_n)}}^{\textrm{core}}}$ is 
\begin{center}
    $\mathscr{F}_{[\sigma_n]_{(d,0)}}:\mathbf{STRUCT}[\sigma_n]_{(d,0)}^{\widetilde {\mathcal{T}(\sigma_n)}}\rightarrow\mathbf{STRUCT}[\sigma_n]_{(d,0)}^{\widetilde{\mathcal{T}(\sigma_n)}}/\equiv_{\rightarrow}.$ 
\end{center}
Thus, $w\mathbf{STRUCT}[\sigma_n]_{(d,0)}^{\widetilde{\mathcal{T}(\sigma_n)}}$ is the inverse image under $\mathscr{F}_{[\sigma_n]_{(d,0)}}$ of $\mathrm{iso}\mathbf{STRUCT}[\sigma_n]_{(d,0)}^{\widetilde{\mathcal{T}(\sigma_n)}}/\equiv_{\rightarrow}$; that is, the inverse image under $\mathscr{F}_{[\sigma_n]_{(d,0)}}$ of $\mathscr{C}_{[\sigma_n]_{(d,0)}}$; and $\tilde{f} \mathbf{STRUCT}[\sigma_n]_{(d,0)}^{\widetilde{\mathcal{T}(\sigma_n)}}$ is defined as the subcategory of retractions in $\mathbf{STRUCT}[\sigma_n]_{(d,0)}^{\widetilde{\mathcal{T}(\sigma_n)}}$.

\begin{pps}
Every object is fibrant and cofibrant in the generalized core model structure on $\mathbf{STRUCT}[\sigma_n]_{(d,0)}^{\widetilde{\mathcal{T}(\sigma_n)}}$.
\end{pps}

\begin{proof}
Since all morphisms $\widetilde{\mathcal{T}(\sigma_n)} \rightarrow N_{d}^{\mathfrak{A}}(\vec{a})$ in $\mathbf{STRUCT}[\sigma_n]_{(d,0)}^{\widetilde{\mathcal{T}(\sigma_n)}}$ are cofibrations, all $\sigma_n$-structures are cofibrants in $\mathbf{STRUCT}[\sigma_n]_{(d,0)}^{{\widetilde{\mathcal{T}(\sigma_n)}}^{\textrm{core}}}$. Proof that all $\sigma_n$-structures are fibrants in $\mathbf{STRUCT}[\sigma_n]_{(d,0)}^{{\widetilde{\mathcal{T}(\sigma_n)}}^{\textrm{core}}}$  is given showing that a morphism $g$ of a $\sigma_n$-structure $N_{d}^{\mathfrak{A}}(\vec{a})$ for the terminal object in $\mathbf{STRUCT}[\sigma_n]_{(d,0)}^{\widetilde{\mathcal{T}(\sigma_n)}}$ is a fibration. But, this follows from Lemma 20.
\end{proof}

\begin{cor}
The homotopy category of $\mathbf{STRUCT}[\sigma_n]_{(d,0)}^{\widetilde{\mathcal{T}(\sigma_n)}}$ is $\mathbf{STRUCT}[\sigma_n]_{(d,0)}^{\widetilde{\mathcal{T}(\sigma_n)}} / \equiv_{\rightarrow}$.
\end{cor}

\begin{proof}
Corollary 2.
\end{proof}

\begin{pps}
Any two morphisms with equal domains and codomains are homotopic in the generalized core model structure on $\mathbf{STRUCT}[\sigma_n]_{(d,0)}^{\widetilde{\mathcal{T}(\sigma_n)}}$.
\end{pps}

\begin{proof}
Since the coproduct of an object with itself is a very good cylinder object, any two morphisms are left homotopic. Because of Proposition 9, we do not need to discriminate between left and right homotopies in the generalized core model structure on $\mathbf{STRUCT}[\sigma_n]_{(d,0)}^{\widetilde{\mathcal{T}(\sigma_n)}}$. 
\end{proof}

\begin{pps}
    Homomorphic equivalence in the category $\mathbf{STRUCT}[\sigma_n]_{(d,0)}^{\widetilde{\mathcal{T}(\sigma_n)}}$ coincides with homotopic equivalence in the model structure $\mathbf{STRUCT}[\sigma_n]_{(d,0)}^{{\widetilde{\mathcal{T}(\sigma_n)}}^{\mathrm{core}}}$.
\end{pps}

\begin{proof}
Follows from Lemma 14.
\end{proof}

\begin{dfn}[Weak $k$-equivalence]
Let $\mathbf{STRUCT}[\sigma_n]_{(d,0)}^{{\widetilde{\mathcal{T}(\sigma_n)}}^{\mathrm{core}}}$ be the core model structure on $\mathbf{STRUCT}[\sigma_n]_{(d,0)}^{\widetilde{\mathcal{T}(\sigma_n)}}$.
A weak $k$-equivalence in $\mathbf{STRUCT}[\sigma_n]_{(d,0)}^{{\widetilde{\mathcal{T}(\sigma_n)}}^{\mathrm{core}}}$ is a weak equivalence between two $\sigma_n$-structures $N_{d}^{\mathfrak{A}}(\vec{a})$ and $N_{d}^{\mathfrak{B}}(\vec{b})$ of $\mathbf{STRUCT}[\sigma_n]_{(d,0)}^{\widetilde{\mathcal{T}(\sigma_n)}}$ which are $k$-homomorphically equivalent.
\end{dfn}

\noindent In other words, a weak $k$-equivalence in $\mathbf{STRUCT}[\sigma_n]_{(d,0)}^{{\widetilde{\mathcal{T}(\sigma_n)}}^{\mathrm{core}}}$ is a morphism in $\mathbf{STRUCT}[\sigma_n]_{(d,0)}^{{(\widetilde{\mathcal{T}(\sigma_n)})}^k}$ which induces isomorphisms in $\mathbf{STRUCT}[\sigma_n]_{(d,0)}^{{(\widetilde{\mathcal{T}(\sigma_n)})}^k} / \equiv_{\rightarrow^k}$.

\begin{pps}
    $k$-Homomorphic equivalence in the category $\mathbf{STRUCT}[\sigma_n]_{(d,0)}^{\widetilde{\mathcal{T}(\sigma_n)}}$ coincides with $k$-homotopic equivalence in the model structure $\mathbf{STRUCT}[\sigma_n]_{(d,0)}^{{\widetilde{\mathcal{T}(\sigma_n)}}^{\mathrm{core}}}$.
\end{pps}

\begin{proof}
Follows from Proposition 11 and the definition of $k$-homomorphism. 
\end{proof}




\begin{thm}[MAIA]
	There is a Quillen model structure $\mathbb{M}$ on $\mathbf{STRUCT}[\sigma_n]_{(d,0)}^{\widetilde{\mathcal{T}(\sigma_n)}}$ such that the homotopic equivalences in $\mathbb{M}$ coincides with the homomorphic equivalences in $\mathbf{STRUCT}[\sigma_n]_{(d,0)}^{\widetilde{\mathcal{T}(\sigma_n)}}$, and such that for every $k$-homotopic equivalence, $\sim_k$, and every $k$-logical equivalence, $\equiv_k$, $N_{d}^{\mathfrak{A}}(\vec{a}) \sim_k N_{d}^{\mathfrak{B}}(\vec{b})$ if and only if $N_{d}^{\mathfrak{A}} (\vec{a}) \equiv_k N_{d}^{\mathfrak{B}}(\vec{b})$, for every primitive positive sentence with quantifier-rank $k$.
\end{thm}

\begin{proof}
The model structure is $\mathbf{STRUCT}[\sigma_n]_{(d,0)}^{{\widetilde{\mathcal{T}(\sigma_n)}}^{\mathrm{core}}}$. That homotopic equivalences
in $\mathbf{STRUCT}[\sigma_n]_{(d,0)}^{{\widetilde{\mathcal{T}(\sigma_n)}}^{\mathrm{core}}}$ coincides with the homomorphic equivalences in
$\mathbf{STRUCT}[\sigma_n]_{(d,0)}^{{\widetilde{\mathcal{T}(\sigma_n)}}^{\mathrm{core}}}$ has already been shown in Proposition 11. That the $k$-homotopic equivalence relation coincides with the logical $k$-equivalence relation, for every primitive-positive sentence with quantifier-rank $k$, follows from Lemma 28 and Proposition 12.
\end{proof}


The above result allows you to define locality under logical equivalence without game-based frameworks. That is, different from what happens with the approach of Arenas, Barceló and Libkin, who start from a game-based framework (game-based locality), that is, describe the logical indistinguishability of neighborhoods in terms of $\mathfrak{F}$-games, the approach proposed here is that of a Quillen model categories-based framework (locality under $k$-homotopic equivalence, for some $k$), that is, the purpose here is to describe logical indistinguishability of neighborhoods in terms of homotopic notions. This is interesting not only because it is an alternative to the game-based framework, but also because it opens up a new range of possibilities for working with locality under logical equivalence, namely the whole technical apparatus that comes up with Quillen model categories.

Although Theorem 3 remains valid only for primitive-positive sentences, it is valid for all sentences if we consider only a special class of structures.

The notation $(\mathfrak{A},\vec{a}) \rightarrow_X (\mathfrak{B},\vec{b})$ (introduced in [21] \S 2.2) to express that there exists a homomorphism from $\mathfrak{A}$ to $\mathfrak{B}$ over $X$ which carries tuple $\vec{a}$ to tuple $\vec{b}$. The notation extends to $k$-homomorphism over $X$ in the obvious way.

\begin{dfn}
A structure $\mathfrak{A}$ is $k$-extendable if, for every set $X \subseteq A$ of size $< k$ and every structure $\mathfrak{B}$ such that $\mathfrak{A} \rightleftarrows_X^{k-|X|} \mathfrak{B}$, it holds that $\forall b \in B \textrm{ } \exists a \in A$ s.t. $\mathfrak{A} \rightleftarrows_X^{k-|X|-1} \mathfrak{B}$.
\end{dfn}{}

\begin{lema}
Suppose structure $\mathfrak{A}$ and $\mathfrak{B}$ are $k$-extendable and $\mathfrak{A} \rightleftarrows^n \mathfrak{B}$. Then $\mathfrak{A} \equiv_k \mathfrak{B}$.
\end{lema}{}

\begin{proof}
([21], p. 30).
\end{proof}{}

\begin{cor}
For $k$-extendables $\sigma_n$-structures, Theorem 3 holds for every sentence with quantifier-rank $k$.
\end{cor}{}

\section{Final Considerations}

Throughout this paper I have presented the implications of a Quillen model category-based framework for locality under logical equivalence. However, one point of my proposal remains problematic. As noted in Theorem 3, $k$-homotopic equivalence of $d$-neighborhoods only implies $k$-logical equivalence for primitive-positive sentences of quantifier-rank $k$. That is, $k$-homotopic equivalence of $d$-neighborhoods does not imply $k$-logical equivalence of $d$-neighborhoods for every sentence of quantifier-rank $k$.

So my goal in future developments is to extend $k$-homotopic equivalence to imply not only $k$-logical equivalence for primitive-positive sentences of quantifier-rank $k$, but to imply $k$-logical equivalence for every sentence of quantifier-rank $k$. In addition, it is of obvious interest to investigate the behavior of the bi-implication ''$k$-homotopic equivalence $\Leftrightarrow$ $k$-logical equivalence'' in logics other than FO.

It is also possible to focus on the definition of model structures over $\mathbf{STRUCT}[\sigma_n]_{(d,0)}^{\widetilde{\mathcal{T}(\sigma_n)}}$ in order to investigate the properties of their homotopic equivalences with respect to locality. For example, there are three trivial model structures over $\mathbf{STRUCT}[\sigma_n]_{(d,0)}^{\widetilde{\mathcal{T}(\sigma_n)}}$, where the choice of subcategories of fibrations, cofibrations, and weak equivalences are reduced to $\mathbf{STRUCT}[\sigma_n]_{(d,0)}^{\widetilde{\mathcal{T}(\sigma_n)}}$ and $\mathbf{STRUCT}[\sigma_n]_{{(d,0)}^{\widetilde{\mathcal{T}(\sigma_n)}}_{(d,0)_{\mathrm{iso}}}}$ (its restriction to isomorphisms). In the case where we have a model structure $\mathbb{M}$ over $\mathbf{STRUCT}[\sigma_n]_{(d,0)}^{\widetilde{\mathcal{T}(\sigma_n)}}$ whose subcategory of weak equivalences is $\mathbf{STRUCT}[\sigma_n]_{{(d,0)}^{\widetilde{\mathcal{T}(\sigma_n)}}_{(d,0)_{\mathrm{iso}}}}$, trivially follows that the locality under weak equivalences is only the usual locality under isomorphisms. Thus, it is possible to classify and investigate locality under different equivalences by investigating possible model structures over $\mathbf{STRUCT}[\sigma_n]_{(d, 0)}^{\widetilde{\mathcal{T}(\sigma_n)}}$.

\bibliographystyle{unsrt}  


\end{document}